\documentclass[review]{elsarticle}

\usepackage{amsmath,amssymb,amsopn,cases,xcolor,accents}
\usepackage{subcaption}
\usepackage{lineno,hyperref}
\usepackage{epstopdf}
\modulolinenumbers[5]

\newtheorem{assumption}{Assumption}
\newtheorem{problem}{Problem}
\newtheorem{lemma}{Lemma}
\newtheorem{remark}{Remark}
\newtheorem{theo}{Theorem}
\newtheorem{coro}{Corollary}
\newenvironment{proof}[1][]{\noindent\textbf{Proof #1:~}}{\hfill$\square$\\}

\newcommand{\A}{\mathcal{A}_{v_{ref}}}
\newcommand{\D}{\mathcal{D}_{v_{ref}}}
\renewcommand{\L}{\mathcal{L}}

\newcommand{\R}{\mathbb{R}}
\renewcommand{\S}{\mathbb{S}}

\DeclareMathOperator*{\col}{col}
\DeclareMathOperator*{\He}{He}
\DeclareMathOperator*{\Sign}{Sign}
\DeclareMathOperator*{\sign}{sign}

\journal{System and Control Letters}

\begin{document}
	
	\begin{frontmatter}
		
		\title{Lyapunov Stability Analysis of a Mass-Spring system subject to Friction}
		\tnotetext[mytitlenote]{This work was supported in part by the ANR project HANDY contract number 18-CE40-0010.}
		
		
		\author[KTH]{Matthieu Barreau\corref{mycorrespondingauthor}}
		\cortext[mycorrespondingauthor]{Corresponding author}
		\ead{barreau@kth.se}
		
		\author[LAAS]{Sophie Tarbouriech}
		\ead{tarbouriech@laas.fr}
		
		\author[LAAS]{Fr\'ed\'eric Gouaisbaut}
		\ead{fgouaisb@laas.fr}
		
		\address[KTH]{Division of Decision and Control Systems, KTH Royal Institute of Technology Stockholm, Sweden.}
		\address[LAAS]{LAAS-CNRS, Universit\'e de Toulouse, CNRS, UPS, Toulouse, France.}
		
		\begin{abstract}
			This paper deals with the stability analysis of a mass-spring system subject to friction using Lyapunov-based arguments. As the described system presents a stick-slip phenomenon, the mass may then periodically sticks to the ground. The objective consists of developing numerically tractable conditions ensuring the global asymptotic stability of the unique equilibrium point. The proposed approach merges two intermediate results: The first one relies on the characterization of an attractor around the origin, to which converges the closed-loop trajectories. The second result assesses the regional asymptotic stability of the equilibrium point by estimating its basin of attraction. The main result relies on conditions allowing to ensure that the attractor issued from the first result is included in the basin of attraction of the origin computed from the second result. An illustrative example draws the interest of the approach.
		\end{abstract}
		
		
	\end{frontmatter}
	
	
	\section{Introduction}  
	
	Friction appears in many mechanical systems such as drill-strings \cite{azar2007drilling,canudasdewit:hal-00394990}, car steering \cite{lozia2002vehicle} or also machine positioning \cite{armstrong1994survey,beerens2019reset}. This is a nonlinear force that is responsible for many undesirable effects such as stick-slip or hunting \cite{armstrong1994survey}. The first challenge was then to propose a model that was generic enough to be easily adapted to a new situation and able to reproduce these phenomena.
	
	Many models of different complexities were investigated in the last century with a relatively good correlation with empirical data. The first model was proposed by Guillaume Amontons and Charles Augustin de Coulomb \cite{coulomb1821theorie} during the eighteenth century. Then, it was studied more precisely by Richard Stribeck \cite{stribeck1902wesentlichen} who experimentally observed a decrease of the friction force at low velocity. Then many models arose trying to fit with the experiments conducted by Stribeck such as the Dahl model \cite{dahl1968solid}, the LuGre model \cite{de1995new} or Leuven friction model \cite{swevers2000integrated}. The survey \cite{armstrong1994survey} and the paper \cite{liu2015experimental} draw comparisons on several models from simulation and experimental points of view.
	
	Before designing controllers (see \cite{beerens2019reset,bisoffi2017global,putra2007analysis}) which were able to reduce the undesirable effects, it was needed to provide analysis tools to quantify the amplitudes of the induced oscillations \cite{kh2004stability}. This challenge gave rise to many techniques quantifying the nonlinear behavior introduced by the nonlinear term. For an empirical analysis, one can refer to \cite{armstrong1990stick}. An approximation based method like the describing method is studied in \cite{canudasdewit:hal-00394990,mcmillan1997non}. An investigation of the analytical solution is conducted in \cite{thomsen2003analytical} and Lyapunov methods are used in \cite{abdo2011}. Few of the previously cited works are providing an exact stability test that is reliable with a low computational burden, and, to the best of our knowledge, there does not exist any regional (local) stability analysis.
	
	This paper deals with the stability analysis of a mass-spring system subject to friction taking into account the the stick-slip phenomenon, meaning that the mass periodically may stick to the ground.
	This paper focuses on deriving numerically tractable conditions ensuring the global asymptotic stability of the origin.
	The approach proposed to attain this objective follows an alternative route to the one developed in \cite{bisoffi2017global}, \cite{zaccarianARC2020}. Furthermore, it is important to emphasize that we do not use the notion of dissipated energy by friction \cite{leine2007stability} to design the Lyapunov function, preferring the use of quadratic Lyapunov functions to simplify the developments. Indeed, the idea is to combine two intermediate results, the first one relying on a global convergence property to an attractor around the equilibrium point, whereas the second one focuses on the estimation of the basin of attraction of the equilibrium point. Hence, the first result concerns the characterization of an attractor around the equilibrium point, in which converges the closed-loop trajectories, by using Lyapunov-based arguments. The second result studies the regional asymptotic stability of the equilibrium point and proposes an estimation of its basin of attraction. The main result expands the two sets of previous conditions to ensure that the attractor issued from the first result is included in the basin of attraction of the origin computed from the second result. Moreover, we provide conditions related to the system physics to characterize the case when the conditions for global asymptotic stability of the origin are feasible. An illustrative example shows the key strengths and the drawbacks of the proposed technique. 
	
	The paper is organized as follows. Section \ref{sec:model-pb} is devoted to the description of the physical setup and the problem statement. In Section \ref{sec:attractor}, the characterization of the attractor around the origin is presented. In Section~\ref{sec:basin}, an inner-estimate of the basin of attraction of the origin is derived. The main result dealing with the global asymptotic stability of the origin is given in Section \ref{sec:GAS}. Section \ref{sec:ex} illustrates the effectiveness of the proposed approach. Finally, Section \ref{sec:conclu} ends the paper emphasizing possible perspectives.
	
	\noindent{\bf Notation.} $\R^{n \times m}$ stands for the set of all $n \times m$ real matrices. $P \in \S_+^n$ or equivalently $P \succ 0$ denotes a symmetric positive-definite matrix of $\R^{n \times n}$. For any matrix $A$, $A^\top$ denotes its transpose. When the matrix is square, we define the operation $\He(A) = A + A^{\top}$. The notation $I_n$ denotes the $n$ by $n$ identity matrix. We define the operation $\col(u, v) = \left[ \begin{matrix} u^{\top} & v^{\top} \end{matrix} \right]^{\top}$ for any column vectors $u$ and $v$. The Euclidean norm of a column-vector $u \in \R^n$ is $\| u \| = \sqrt{ u^{\top} u }$.
	The symbol $\langle \cdot, \cdot \rangle$ denotes the standard Euclidean inner product.
	
	\section{Model description and problem formulation}
	\label{sec:model-pb}
	\subsection{Physical setup}
	
	The setup, already introduced in \cite{johanastrom2008revisiting}, consists of studying the position and the velocity of a moving body of mass $m$, connected to a spring with an end moving forward at a constant speed $\dot{x}_A = v_{ref}$, parallel to the direction of movement. This is described in Figure~\ref{fig:physic}. 
	
	\begin{figure}[!hhh!]
		\centering
		\includegraphics[width=8cm]{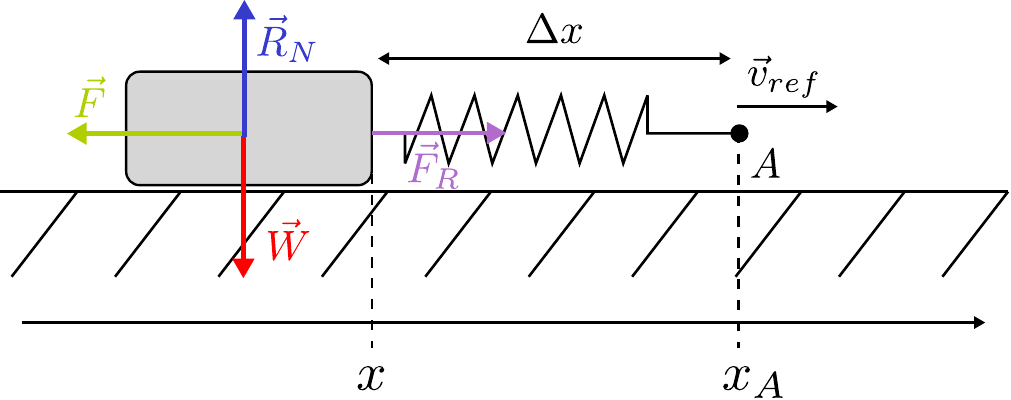}
		\caption{Physical setup: mass with a spring. The forces applied on the mass are in color and the point $A$ has the speed $v_{ref}$.}
		\label{fig:physic}
	\end{figure}
	
	The forces applied to the mass are the following ones:
	\begin{enumerate}
		\item the weight $\vec{W} = m \vec{g}$ where $\vec{g}$ is the gravitational acceleration;
		\item the normal force $\vec{R}_N =  - m \vec{g}$;
		\item the elastic force $\vec{F}_R = k \left( x_A - x - \ell_0 \right)$ with $\ell_0 \in \R$;
		\item and the friction force $\vec{F}$.
	\end{enumerate}
	
	In this work, the simplest friction model is considered. It was introduced firstly by Karnopp \cite{karnopp1985computer} and has shown a good correlation with the experimental data. It was then investigated more deeply in \cite{armstrong1990stick}, and its final mathematical formulation is as follows:
	\begin{eqnarray}
	F(\dot{x}(t)) \!\!\!&=&\!\!\! R_N \left( \mu_C + (\mu_S - \mu_C) e^{-\left| \frac{\dot{x}(t)}{v_s} \right|^{2}} \right) \Sign(\dot{x}(t)) + k_v \dot{x}(t), \nonumber\\
	\!\!\!&=&\!\!\! F_{nl}(\dot{x}(t)) + k_v \dot{x}(t), \label{eq:friction}
	\end{eqnarray}
	where $R_N = mg$, $v_s$ is a positive constant, $\mu_S$ and $\mu_C$ positive constants such that $\mu_S - \mu_C >0$, and the $\Sign$ function is defined as:
	\begin{equation}
	\Sign(\theta) = \left\{ \begin{array}{ll} 
	\sign(\theta) = \frac{\theta}{|\theta|} & \text{ if } \theta \neq 0, \\
	{[-1, 1]} & \text{ if } \theta = 0,
	\end{array} \right.
	\label{eq:defsign}
	\end{equation}
	for $\theta \in \R$. Note that the $\Sign$ is defined as a set-valued function as in \cite{beerens2019reset,bisoffi2017global} and can be seen as the convex hull of the classical $\sign$ function. Note that it is fundamental to consider the set-valued friction force since it induces the Coulomb's cone of friction which is one of the most important characteristics in frictional systems \cite{acary2008numerical}.
	
	The previous expression for the friction force can be split into three important contributions:
	\begin{enumerate}
		\item the dynamic Coulomb force: $\mu_C \left( 1 - e^{-\frac{\dot{x}^2(t)}{v_s^2}} \right) R_N \text{Sign}(\dot{x}(t))$;
		\item the static friction force: $\mu_S e^{-\frac{\dot{x}^2(t)}{v_s^2}} R_N \text{Sign}(\dot{x}(t))$;
		\item and the viscous friction: $k_v \dot{x}(t)$.
	\end{enumerate}
	The Coulomb force is the friction force acting at relatively high speed while the static friction occurs for low velocity. The viscous friction is a linear term that is related mostly to the friction with the air. A chart of this function is proposed in Figure~\ref{fig:friction} to observe the Stribeck curve, that is the non-monotonicity of the function $F$ for positive speed. One novelty of this study is to take into account this phenomenon which is sometimes not included \cite{leine2007stability,van2004attractivity}.
	
	\begin{figure}[!hhh!]
		\centering
		\includegraphics[width=8cm]{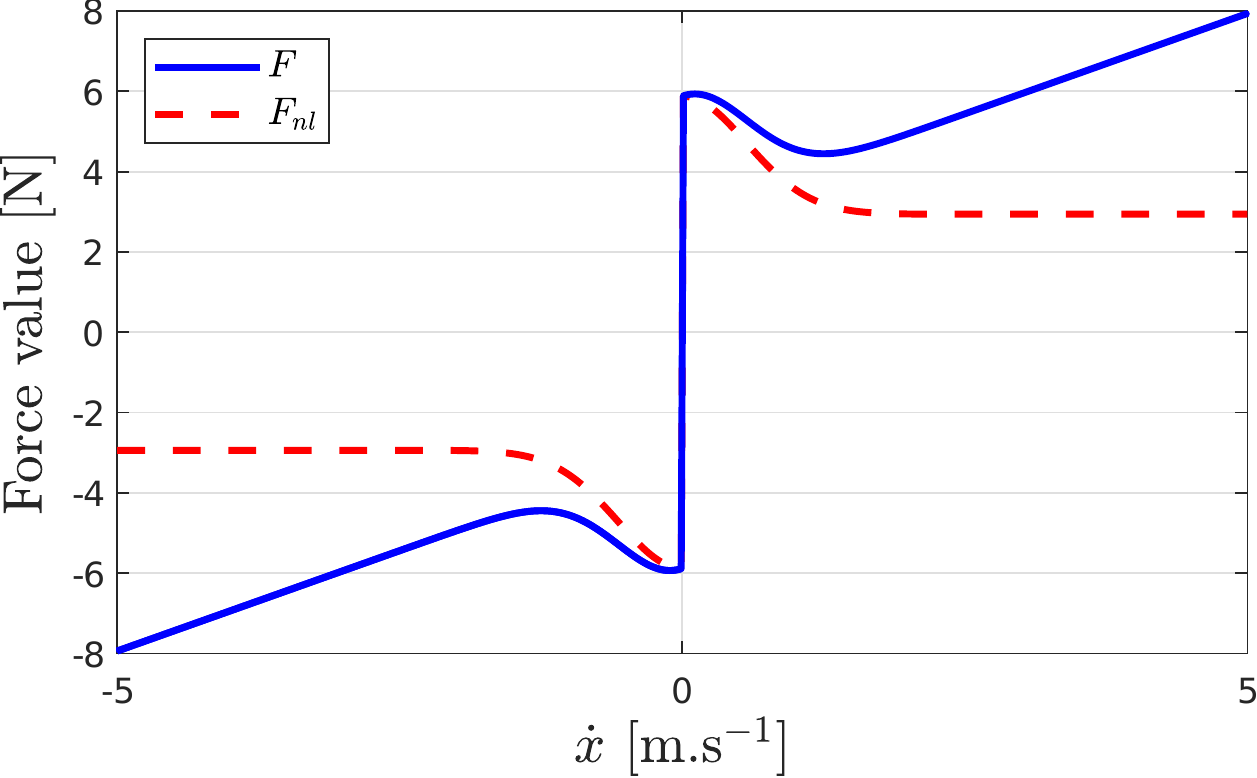}
		\caption{Chart of the friction force (in solid blue) and static and Coulomb forces (in dashed red) in our problem. Note that the friction force is not a monotonous function of the speed, inducing the famous stick-slip effect.}
		\label{fig:friction}
	\end{figure}

	We study here the evolution of $\dot{x}$ and more particularly, its asymptotic behavior. For $x_A(0) \in \R$, by applying the second law of motion, we obtain the following dynamical equation for $t \geq 0$:
	\begin{equation} \label{eq:motion}
	\left\{
	\begin{array}{l}
	\ddot{x}(t) \in - \frac{1}{m}F(\dot{x}(t)) - \frac{k}{m} \left( x(t) + \ell_0 - x_A(0) - v_{ref} t \right), \\
	x(0) = 0, \quad \dot{x}(0) = v_0,
	\end{array}
	\right.
	\end{equation}
	where $v_0 \in \mathbb{R}$.
	
	\subsection{Problem Statement}
	
	Without loss of generality, we assume $x_A(0) = 0$ and introduce the following state variables: $v = \dot{x}$ and $z(t) = x(t) + \ell_0 - v_{ref} t$. \eqref{eq:motion} reads then for $t > 0$ as:
	\begin{equation} \label{eq:stateSpace}
	\left[ \begin{matrix} \dot{v}(t) \\ \dot{z}(t) \end{matrix} \right] \in \left[ \begin{matrix} - \frac{1}{m} \left( F_{nl}(v(t)) + k_v v(t) + k z(t) \right) \\ v(t) - v_{ref} \end{matrix} \right],
	\end{equation}
	with $v(0) = v_0$, $z(0) = z_0$.
	
	\begin{remark}
		Using a similar reasoning as in \cite{bisoffi2017global,zaccarianARC2020,filippov2013differential}, one can conclude that there exists a unique solution to system~\eqref{eq:stateSpace}.
	\end{remark}
	
	We make the two following assumptions.
	\begin{assumption}\label{ass1}
		We assume $v_{ref} > 0$.
	\end{assumption}
	\begin{assumption}\label{ass2}
		The viscosity of the air $k_v$ and the stiffness of the spring $k$ are strictly positive.
	\end{assumption}
	
	Note first that Assumption \ref{ass1} implies that $v_{ref}$ is not zero. We are not seeking to control the position as in \cite{beerens2019reset} but to keep the velocity constant. This is desirable in the drilling industry \cite{azar2007drilling,drillingArticle} where the bit should rotate at a constant angular speed. There are other examples such as speed controllers in cars \cite{lozia2002vehicle,johanastrom2008revisiting} or for tracking purposes in telescopes \cite{gawronski2007control}. To ease the calculations, we assume that $v_{ref}$ is positive but a similar analysis can be conducted with $v_{ref} < 0$. 
	Concerning Assumption \ref{ass2}, this is physically reasonable since it is related to the stability of the error system for large speed as we will see in the sequel.
	
	We are now looking for the equilibrium points $(v_{\infty}, z_{\infty})$ of \eqref{eq:stateSpace}. It makes sense to look for an equilibrium only in the zone where $F$ is a singleton. Due to Assumptions~\ref{ass1} and \ref{ass2}, there is a unique equilibrium point and we easily get:
	\begin{equation}
	\left\{ \begin{array}{l}
	\displaystyle v_{\infty} = v_{ref}, \\
	\displaystyle z_{\infty} = - \frac{F(v_{ref})}{k} = - \frac{F_{nl}(v_{ref}) + k_v v_{ref}}{k}.
	\end{array} \right.
	\end{equation}
	By defining the following:
	\begin{equation}\label{eq:vareps+phi}
	\varepsilon = \left[ \begin{matrix} \varepsilon_1 \\ \varepsilon_2 \end{matrix} \right] = \left[ \begin{matrix} v - v_{ref} \\ z - z_{\infty} \end{matrix} \right],
	\quad 
	\phi_{v_{ref}}(\varepsilon_1) = F_{nl}(\varepsilon_1 + v_{ref}) - F_{nl}(v_{ref}),
	\end{equation}
	the error dynamic is described by the dynamical system:
	\begin{equation} \label{eq:error}
	\dot{\varepsilon}(t) \in A \varepsilon(t) +  B \phi_{v_{ref}}(\varepsilon_1(t))
	\end{equation}
	with $\varepsilon_1(0) = v_0 - v_{ref}$,  $\varepsilon_2(0) = z_0 - z_{\infty}$, 
	\begin{equation} \label{eq:A}
	A = \left[\begin{matrix} - \frac{k_v}{m} & - \frac{k}{m} \\ 1 & 0 \end{matrix} \right] \quad \text{ and } \quad B = \left[ \begin{matrix} - \frac{1}{m} \\ 0 \end{matrix} \right].
	\end{equation}
	
	In the case $\phi_{v_{ref}} = 0$, from Assumption~2, the error system is stable. However, from practical experiments, one can observe that due to the presence of the nonlinearity $\phi_{v_{ref}}(\varepsilon_1)$, it can be impossible to guarantee the global asymptotic convergence of the trajectories to the origin. Some limit cycles around the origin can exist \cite{armstrong1990stick,drillingArticle}. Then, for the class of systems as described by \eqref{eq:error}, it is of interest to characterize a region $\A$ of the state space, containing the origin and the possible limit cycle, which is guaranteed to be a global attractor of system \eqref{eq:error}.
	Then the first problem we intend to address is the following.
	\begin{problem} Characterize an outer-estimation of the globally stable attractor $\A$ of system \eqref{eq:error}.
		\label{pb:attractor}
	\end{problem}
	
	For some cases, one can observe that the trajectories of system \eqref{eq:error} converge to the origin. Then, it appears possible to estimate the region of initial conditions for which the asymptotic stability of the origin will be obtained. In this sense, the second problem we are interested in relies on regional asymptotic stability of the origin. Note that since $\phi_{v_{ref}}$ is continuous around $\phi_{v_{ref}}(0) = 0$, there might exist a basin of attraction around the origin. But to be able to estimate it we need to study what happens for the linearized version of system \eqref{eq:error}, and therefore what happens with the derivative of $\phi$ with respect to $\varepsilon_1$ denoted $\frac{\partial  \phi}{\partial \varepsilon_1}$ \cite{khalil1996nonlinear}.
	Then the second problem we intend to address is the following.
	\begin{problem} Characterize an inner-estimation of the basin of attraction of the origin, denoted $\D$ for system \eqref{eq:error}.
		\label{pb:regional}
	\end{problem}
	Unlike Problem \ref{pb:attractor}, to address this problem, we need to find conditions depending on $v_{ref}$.
	
	Finally, the main problem we want to solve is to provide conditions to ensure that the origin will be a globally asymptotically stable equilibrium point. To do this, the approach will consist of combining the solutions to Problems \ref{pb:attractor} and \ref{pb:regional}. Hence, if we can provide conditions such that $\A \subseteq \D$, then there does not exist a limit cycle and the origin will be globally asymptotically stable. 
	Thus the main problem we intend to address is the following.
	\begin{problem} Derive conditions for the outer-estimation of the attractor to be included in the inner-estimation of the basin of attraction of the origin, that is, $\A \subseteq \D$. 
		\label{pb:main}
	\end{problem}
	
	\section{Global Attractor Characterization}
	\label{sec:attractor}
	
	In order to address Problem~\ref{pb:attractor}, one first proposes a way to encapsulate the nonlinearity $\phi_{v_{ref}}$. Since $F_{nl}$ is an odd function and $F_{nl}^2$ is lower and upper bounded almost everywhere by $F_C^2 = \mu_C^2 m^2 g^2$ and $F_S^2 = \mu_S^2 m^2 g^2$ respectively as shown in Figure~\ref{fig:relay}, the function $F_{nl}$  can be encapsulated into two relays.
	\begin{figure}[!hhh!]
		\centering
		\includegraphics[width=8.5cm]{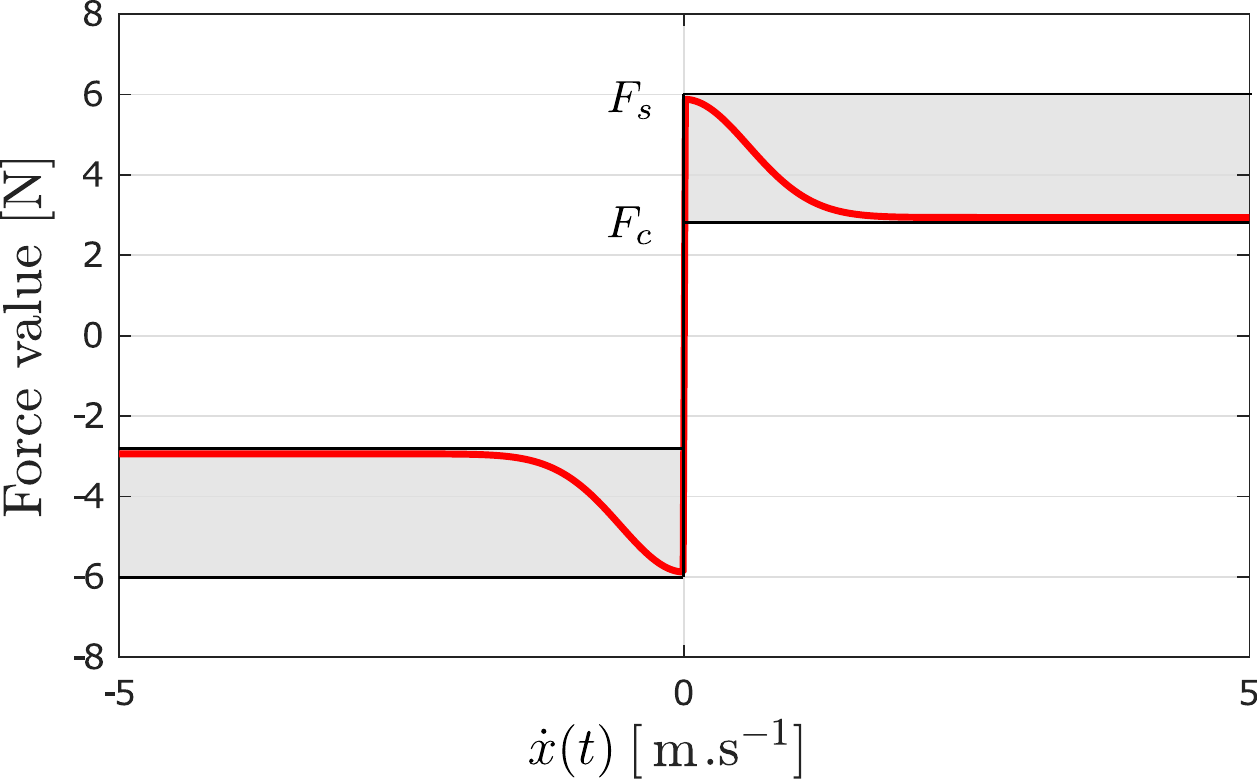}
		\caption{Encapsulation of $F_{nl}$.}
		\label{fig:relay}
		\vspace*{-0.45cm}
	\end{figure}
	
	Consequently, the following lemma is proposed.
	\begin{lemma} \label{lem:relay} For any $\varepsilon_1 \in \R$ the following inequalities hold:
		\begin{equation}
		\begin{array}{rcl}
		F_{nl}(\varepsilon_1 + v_{ref})^2 \!\!\!& \leq \!\!\!& F_S^2, \\
		- 2 (\varepsilon_1 + v_{ref}) F_{nl}(\varepsilon_1 + v_{ref}) \!\!\!& \leq \!\!\!& 0.
		\end{array}
		\end{equation}
		In the case $\varepsilon_1 \neq - v_{ref}$, we also have:
		\begin{equation}
		F_C^2 \leq  F_{nl}(\varepsilon_1 + v_{ref})^2.
		\end{equation}
	\end{lemma} 
	
	
	By using Lemma \ref{lem:relay}, the following theorem proposes a solution to Problem \ref{pb:attractor}.
	
	\begin{theo} \label{theo:global} If there exist $P_g \in \mathbb{S}^2_+$ and $\tau_0, \tau_1, \tau_2, \tau_3, \tau_5 \geq 0$, $\tau_4 \in \R$ such that the following matrix inequalities hold:
		\begin{equation} \label{eq:LMIglobal1}
		\He\left( D^{\top} P_g F \right) - \tau_0 \Pi_0 - \tau_1 \Pi_1 - \tau_2 \Pi_2 - \tau_3 \Pi_3 \prec 0,
		\end{equation}
		\begin{equation} \label{eq:LMIglobal2}
		\He\left( D^{\top} P_g F \right) - \tau_0 \Pi_0 - \tau_5 \Pi_1 - \tau_4 \Pi_4 \prec 0,
		\end{equation}
		where 
		\[
		\begin{array}{ll}
		D = \left[ \begin{array}{cccc} A & B & - F_{nl}(v_{ref}) B \end{array}\right], & 
		F = \left[ \begin{matrix} I_2 & 0 & 0  \end{matrix} \right],\\
		F_C = \mu_C m g, & F_S = \mu_S m g, \\
		\pi_1 = \left[ \begin{matrix} 1 & 0 & 0 & 0 \end{matrix} \right],& \pi_3 =  \left[ \begin{matrix} 0 & 0 & 1 & 0 \end{matrix} \right], \\
		\pi_4 =  \left[ \begin{matrix} 0 & 0 & 0 & 1 \end{matrix} \right], & \\
		\Pi_0 = \pi_4^{\top} \pi_4 - F^{\top} P_g F, & \Pi_1 = \pi_3^{\top} \pi_3 - F_S^2 \pi_4^{\top} \pi_4, \\
		\Pi_2 = F_C^2 \pi_4^{\top} \pi_4 - \pi_3^{\top} \pi_3, &
		\Pi_3 = -\He\left( (\pi_1 + v_{ref} \pi_4)^{\top} \pi_3 \right), \\
		\Pi_4 = (\pi_1 + v_{ref} \pi_4)^{\top}(\pi_1 + v_{ref} \pi_4),
		\end{array}
		\]
		then trajectories of system~\eqref{eq:error} globally converge to the set 
		\[
		\A(P_g) = \left\{ \varepsilon \in \mathbb{R}^2 \ | \ \varepsilon^{\top} P_g \varepsilon < 1 \right\}.
		\]
		Hence, the set $\A(P_g)$ is a solution to Problem \ref{pb:attractor}.
	\end{theo}
	
	\begin{proof} The proof of this theorem is based on the following Lyapunov function:
		\begin{equation}
		V_g(\varepsilon) = \varepsilon^{\top} P_g \varepsilon.
		\end{equation}
		with $P_g=P_g^\top \succ 0$. 
		Inspired by \cite{tar:que:pri/ieee2014}, \cite{fer:gou:tar/auto2015}, one wants to verify that there exists a class $\mathcal{K}$ function $\alpha$ such that $\langle \nabla V_g,f \rangle \leq - \alpha(V_g(\varepsilon))$, for all $\varepsilon$ such that $\varepsilon^{\top} P_g \varepsilon \geq 1$ (i.e. for any $\varepsilon \in \mathbb{R}^2 \backslash \A(P_g)$), for any $f \in A \varepsilon(t) +  B \phi_{v_{ref}}(\varepsilon_1(t))$ and for all nonlinearities $F_{nl}$ satisfying Lemma \ref{lem:relay}. In other words, by using the S-variable approach \cite{Svariable}, one wants to ensure that
		$\dot{V}_g(\varepsilon) - \tau_0(1 - \varepsilon^{\top} P_g \varepsilon) \leq - \alpha(V_g(\varepsilon))$, with $\tau_0 > 0$, for any $\varepsilon \in \mathbb{R}^2 \backslash \A(P_g)$, for any $f \in A \varepsilon(t) +  B \phi_{v_{ref}}(\varepsilon_1(t))$ and for all nonlinearities $F_{nl}$ satisfying Lemma \ref{lem:relay}.
		
		In the following, we denote by language abuse $\langle \nabla V_g,f \rangle$ by $ \dot{V}_g(\varepsilon)$ meaning that we compute $\langle \nabla V_g,f \rangle$ for any solution $f$ to the differential inclusion \eqref{eq:error}. Then for all solution to 
		\eqref{eq:error}, one gets: 
		\begin{equation}
		\dot{V}_g(\varepsilon) = \xi_g^{\top} \left( D^{\top} P_g F + F^{\top} P_g D \right) \xi_g,
		\end{equation}
		where 
		$
		\xi_g = \col \left( \varepsilon,  F_{nl}(\varepsilon_1 + v_{ref}), 1 \right)
		$.
		The existence of such a function $\alpha$ is discussed in two different cases.
		\begin{itemize}
			\item \textit{First case:} $\varepsilon_1 \neq - v_{ref}$. \\
			Using the augmented vector $\xi_g$ and the definition of the $\pi_i$'s, the conditions of Lemma \ref{lem:relay} read:
			\[
			\begin{array}{rl}
			\xi_g^\top (\pi_3^\top \pi_3 - F_S^2 \pi_4^\top \pi_4) \xi_g \!\!\!\!& \leq 0, \\
			\xi_g^\top (F_C^2 \pi_4^\top \pi_4 -\pi_3^\top \pi_3) \xi_g \!\!\!\!& \leq 0, \\ 
			- \xi_g\He\left( (\pi_1 + v_{ref} \pi_4)^{\top} \pi_3 \right) \xi_g \!\!\!\!& \leq 0.
			\end{array}
			\]
			Then, by using the compact notation $\Pi_i$'s inspired by the S-variable approach, we define the following function:
			\begin{equation} \label{eq:global1}
			\hspace*{-0.3cm}
			\begin{array}{rl}
			\L\!\!\!&= \dot{V}_g(\varepsilon) - \tau_0(1 - \varepsilon^{\top} P_g \varepsilon) - \tau_1 \xi_g^\top \Pi_1 \xi_g -\tau_2 \xi_g^\top \Pi_2 \xi_g -\tau_3 \xi_g^\top \Pi_3 \xi_g
			\end{array}
			\end{equation}
			for positive scalars $\tau_i$, $i\in\{0,1,2,3\}$.\\
			Consequently, if \eqref{eq:LMIglobal1} is verified then $\L$ in \eqref{eq:global1} is negative-definite for ${\varepsilon_1 \neq - v_{ref}}$ and there exists $\alpha_1 > 0$ such that $\L \leq - \alpha_1 \varepsilon^{\top} \varepsilon$.
			
			\item \textit{Second case:} $\varepsilon_1 = -v_{ref}$.\\
			Using Lemma~\ref{lem:relay} it follows:
			\begin{equation}
			\begin{array}{rl}
			\xi_g^\top (\pi_3^\top \pi_3 - F_S^2 \pi_4^\top \pi_4) \xi_g \!\!\!\!& \leq 0, \\
			\xi_g^\top (\pi_1 + v_{ref} \pi_4)^{\top}(\pi_1 + v_{ref} \pi_4) \xi_g \!\!\!\!& \leq 0, \\ 
			- \xi_g\He\left( (\pi_1 + v_{ref} \pi_4)^{\top} \pi_3 \right) \xi_g \!\!\!\!& = 0.
			\end{array}
			\end{equation}
			Let us define $\L$ for $\varepsilon_1 = - v_{ref}$ as
			\begin{equation} \label{eq:global2}
			\hspace*{-0.3cm}
			\begin{array}{rl}
			\L\!\!\!&= \dot{V}_g(\varepsilon) - \tau_0(1 - \varepsilon^{\top} P_g \varepsilon) - \tau_5 \xi_g^\top \Pi_1 \xi_g -\tau_4 \xi_g^\top \Pi_4 \xi_g
			\end{array}
			\end{equation}
			for the same $\tau_0$ as defined previously and $\tau_4 \in \R, \tau_5 \geq 0$.\\
			Thus, if \eqref{eq:LMIglobal1} is satisfied then $\L$ in \eqref{eq:global2} is negative-definite for ${\varepsilon_1 = - v_{ref}}$ and there exists $\alpha_2 > 0$ such that $\L \leq - \alpha_2 \varepsilon^{\top} \varepsilon$.
		\end{itemize}
		
		If \eqref{eq:LMIglobal1} and \eqref{eq:LMIglobal2} are verified simultaneously then for $\alpha = \min(\alpha_1, \alpha_2) > 0$, we get $\mathcal{L} \ \leq - \alpha \varepsilon^\top \varepsilon$. By definition one gets $\dot{V}_g(\varepsilon) - \tau_0(1 - \varepsilon^{\top} P_g \varepsilon) \leq \mathcal{L}$ and since $\dot{V}_g(\varepsilon) \leq \dot{V}_g(\varepsilon) - \tau_0(1 - \varepsilon^{\top} P_g \varepsilon)$ for any $\varepsilon \in \mathbb{R}^2 \backslash \A(P_g)$, it follows $\dot{V}_g(\varepsilon) \leq - \alpha \varepsilon^\top \varepsilon$ for any $\varepsilon \in \mathbb{R}^2 \backslash \A(P_g)$. 
		That ends the proof. \end{proof}
	
	\begin{remark}
		Following the same strategy as in \cite[Corrolary~2]{drillingArticle}, setting $\tau_2 = \tau_3 = 0$ in \eqref{eq:LMIglobal1} and using Schur complement shows that $A$ is Hurwitz (which is the case due to Assumption~\ref{ass2}), if and only if there exists $P_g$ such that \eqref{eq:LMIglobal1}-\eqref{eq:LMIglobal2} are satisfied.
	\end{remark}
	
	\begin{remark}
		Relations~\eqref{eq:LMIglobal1}-\eqref{eq:LMIglobal2} are quasi-LMIs since there is a nonlinearity due to the product between the scalar $\tau_0$ and the matrix $P_g$. Note that a necessary condition to get \eqref{eq:LMIglobal1}-\eqref{eq:LMIglobal2} is that
		\begin{equation}
		A^{\top} P_g + P_g A + \tau_0 P_g \prec 0.
		\end{equation}
		As explored in \cite{drillingArticle}, it means that the following inequality must hold:
		\begin{equation}
		0 \leq \tau_0 \leq \tau_0^{max} = - 2 \max_{\mu \text{ an eigenvalue of } A} (\Re(\mu)).
		\end{equation}
	\end{remark}
	
	The previous remark emphasizes an efficient way to solve \eqref{eq:LMIglobal1}-\eqref{eq:LMIglobal2}. A solution is to use a grid on $\tau_0 \in [0, \tau_0^{max}]$ and, in this case, \eqref{eq:LMIglobal1}-\eqref{eq:LMIglobal2} turn into LMIs.
	
	To get an outer-estimate close to the real attractor, one can solve various optimization problems (see for instance \cite{tarbouriech2011stability}). Among these techniques, the minimization of the maximal axis of the ellipsoid leads to the following optimisation problem:
	\begin{equation} \label{eq:optAttractor}
	\begin{array}{cc}
	\min & - \eta \\
	\text{subject to} & \eqref{eq:LMIglobal1} \text{ and } \eqref{eq:LMIglobal2} \text{ hold with } P_g \succeq \eta I_2.
	\end{array}
	\end{equation}
	
	\section{Regional Asymptotic Stability Analysis}
	\label{sec:basin}
	
	This section concentrates on a regional analysis to solve Problem \ref{pb:regional}.
	
	\subsection{Stability of the linearized system}
	
	Similarly to the techniques used in presence of saturation or backlash nonlinearities (see, for example, \cite{tar:que:pri/ieee2014,tarbouriech2011stability}), the fact  that $\phi_{v_{ref}}$ is continuous with $\phi_{v_{ref}}(0) = 0$, is not sufficient to study the regional stability of the origin. Actually, in a neighborhood of the origin, the function $\phi_{v_{ref}}$ is perfectly smooth, but that does not imply the existence of a basin of attraction. Furthermore, in standard Coulomb friction system \cite{leine2007stability}, the equilibrium can be a point of set-valuedness of the friction law, but in the context of Assumptions \ref{ass1} and \ref{ass2}, the study of the linearized system is justified by the smoothness of $\phi_{v_{ref}}$.
	Hence, using Assumption~\ref{ass1}, we must consider the stability of its linearization around an equilibrium point to apply the first Lyapunov principle \cite{khalil1996nonlinear}. To this extent, we consider a rewriting of the error dynamics system \eqref{eq:error} as follows:
	\begin{equation} \label{eq:errorlocal}
	\begin{array}{lcl}
	\dot{\varepsilon}(t) \!\!\!& \in &\!\!\! \left(A  +  B  \Gamma_{v_{ref}} C \right) \varepsilon(t) + B \left(\phi_{v_{ref}}(\varepsilon_1(t)) -  \Gamma_{v_{ref}} C \varepsilon(t)\right)
	\end{array}
	\end{equation}
	with $A$, $B$, $\phi_{v_{ref}}$ defined in \eqref{eq:A}, $C = \left[ \begin{smallmatrix} 1 & 0 \end{smallmatrix} \right]$ and 
	\[
	\Gamma_{v_{ref}} = \left. \frac{\partial  \phi_{v_{ref}}}{\partial \varepsilon_1}\right|_{\varepsilon_1 = 0} = - 2 R_N (\mu_S - \mu_C) \frac{v_{ref}}{v_s^2} e^{-\frac{v_{ref}^2}{v_s^2}}.
	\]
	System \eqref{eq:errorlocal} reads:
	\begin{equation} \label{eq:errorlocal2}
	\dot{\varepsilon}(t) \in A_0 \varepsilon(t)  +  B \psi_{v_{ref}}(\varepsilon_1)
	\end{equation}
	with 
	\[
	A_0 = A + B \Gamma_{v_{ref}} C = \left[ \begin{smallmatrix} - \frac{k_v + \Gamma_{v_{ref}}}{m} & - \frac{k}{m} \\ 1 & 0 \end{smallmatrix} \right] \text{ and } \psi_{v_{ref}}(\varepsilon_1) = \phi_{v_{ref}}(\varepsilon_1)- \Gamma_{v_{ref}} \varepsilon_1
	\]
	Consequently, the linearization of system~\eqref{eq:errorlocal} around the origin is:
	\begin{equation} \label{eq:errorLinear}
	\dot{\tilde{\varepsilon}}(t) = A_0 \tilde{\varepsilon}(t)
	\end{equation}
	
	Hence, there exists a basin of attraction $\D$ if and only if $A_0$ Hurwitz, which depends then on $v_{ref}$. The following lemma characterizes the stability of \eqref{eq:errorLinear} as a function of $v_{ref}$.
	\begin{lemma} Matrix $A_0$ is Hurwitz if and only if 
		\begin{equation}\label{condLemA0}
		- k_v + 2 R_N (\mu_S - \mu_C) \frac{|v_{ref}|}{v_s^2} e^{-\frac{v_{ref}^2}{v_s^2}}< 0,
		\end{equation}
		or equivalently, $\theta(v_{ref}) = |v_{ref}| e^{-\frac{v_{ref}^2}{v_s^2}}< \frac{k_v v_s^2}{2 R_N (\mu_S - \mu_C)}$.
		\label{lem:A0}
	\end{lemma}
	
	Although Lemma \ref{lem:A0} is directly related to the eigenvalues of matrix $A_0$ as a function of $v_{ref}$, it is interesting to note that condition \eqref{condLemA0} 
	is related to the monotonicity of the mapping (\ref{eq:errorlocal2}) according to \cite{brogliato2020dynamical}.
	
	
	There are two roots such that $\theta(v_{ref}) = \frac{k_v v_s^2}{2 R_N (\mu_S - \mu_C)}$. By denoting them $v_{ref 1}$ and $v_{ref 2} > v_{ref 1}$, one can exhibit several zones regarding the stability of $A_0$: 
	\begin{enumerate}
		\item $A_0$ is Hurwitz for $v_{ref} \in (0,v_{ref 1}) \cup (v_{ref 2}, \infty)$ and, in virtue of the first Lyapunov principle \cite{khalil1996nonlinear}, there exists a basin of attraction for \eqref{eq:errorlocal2} around the origin;
		\item the origin of \eqref{eq:errorLinear} is exponentially unstable for $v_{ref} \in (v_{ref 1},v_{ref 2})$;
		\item $A_0$ has poles on the imaginary axis for $v_{ref} = v_{ref 1}$ or $v_{ref} = v_{ref 2}$.
	\end{enumerate}
	
	We concentrate on an inner-estimation of the basin of attraction around the origin for $v_{ref} \in (0,v_{ref 1}) \cup (v_{ref 2}, \infty)$.
	
	\subsection{Numerical inner-estimate of $\D$}
	
	
	\begin{figure}
		\centering
		\begin{subfigure}[t]{0.46\linewidth}
			\includegraphics[width=\textwidth]{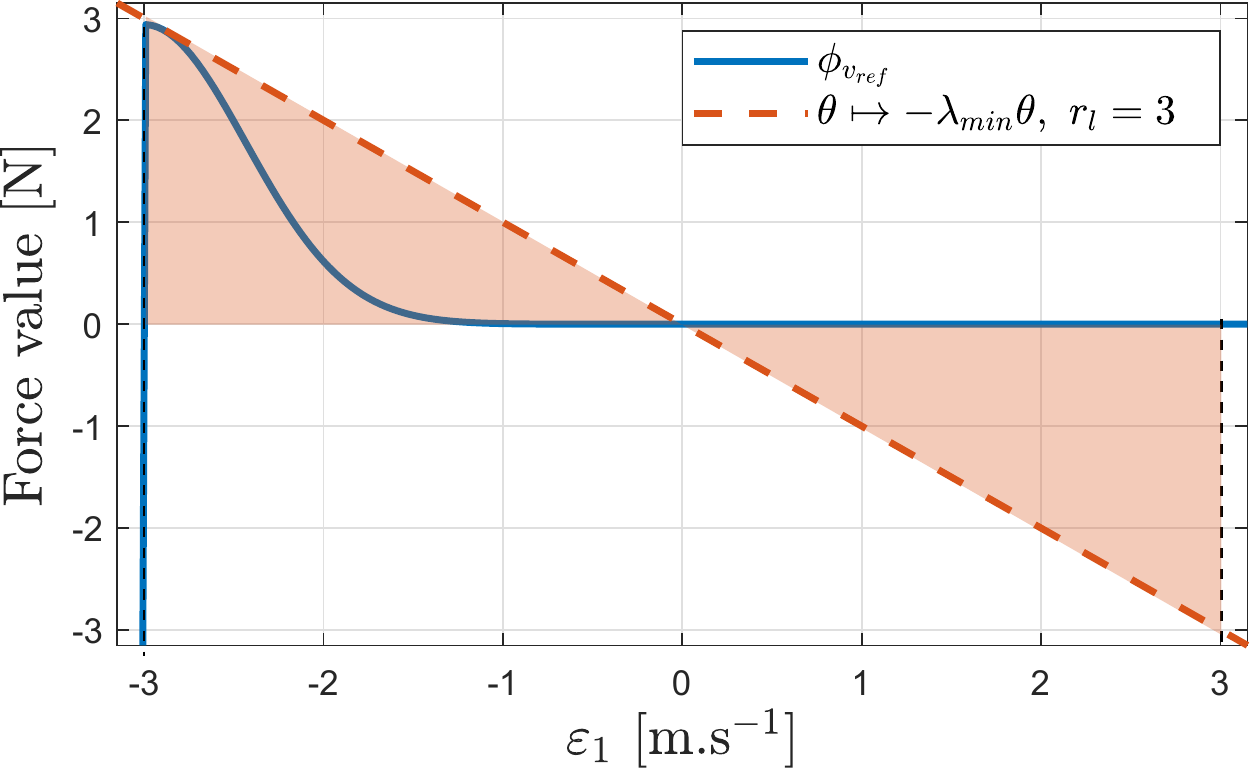}
			\caption{$\lambda_{min} = 1.004$, $r_l = 3$}
			\label{fig:rl3}
		\end{subfigure} \
		\begin{subfigure}[t]{0.46\linewidth}
			\includegraphics[width=\textwidth]{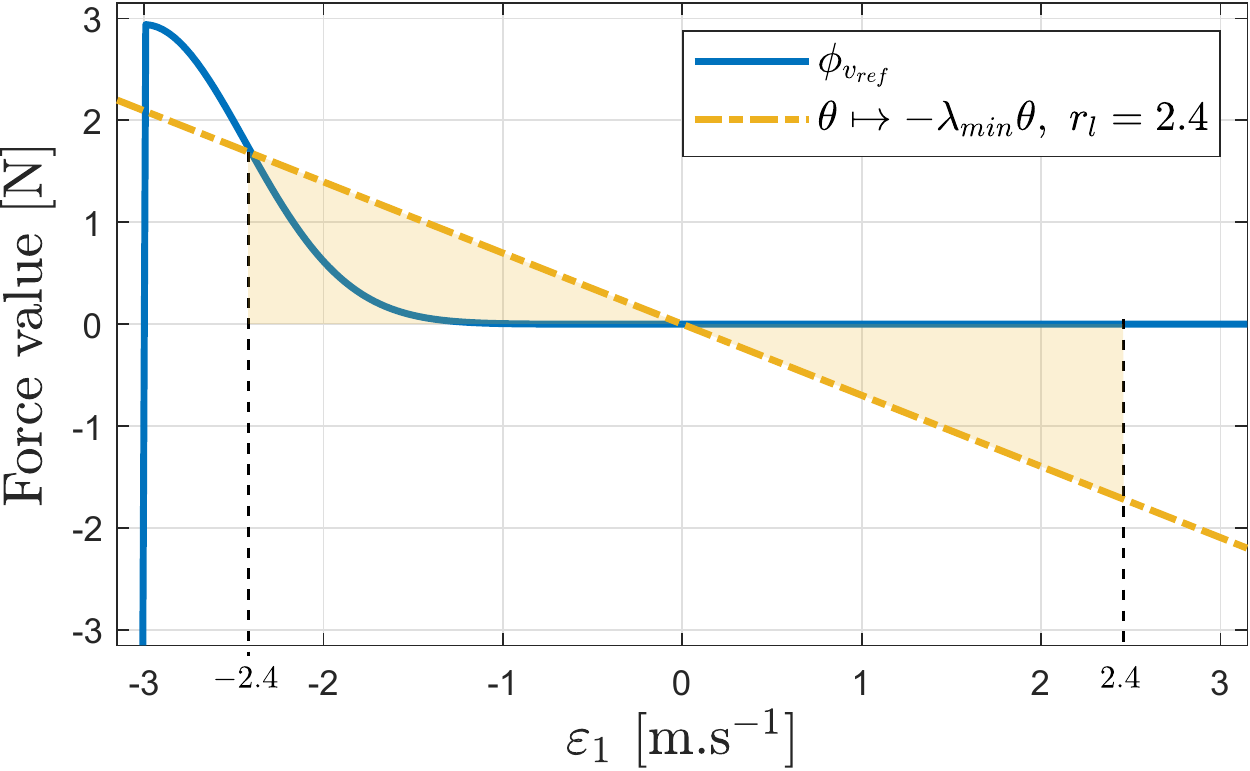}
			\caption{$\lambda_{min} = 0.698$, $r_l = 2.4$}
			\label{fig:rl2.4}
		\end{subfigure}
		\caption{Sector conditions on the function $\phi_{v_{ref}}$ with $v_{ref} = 3$.}
		\label{fig:sectorCondition}
	\end{figure}
	
	Note that $\phi_{v_{ref}}$, defined in \eqref{eq:vareps+phi}, is bounded on $(-v_{ref}, \infty)$ and $\varepsilon_1 \phi_{v_{ref}}(\varepsilon_1) \leq 0$ for $\varepsilon_1 \in (-v_{ref}, \infty)$, consequently, $\phi_{v_{ref}}$ is a local sector-bounded nonlinearity as depicted in Figure \ref{fig:sectorCondition}. A sector condition on $\phi_{v_{ref}}$ is directly related to the fact that system (\ref{eq:errorlocal}) is strictly positive real or equivalently here, strictly passive \cite[Theorem 4.73]{brogliato2007dissipative}, \cite{mad:ada/ieee2016}. However, this property is global and in the current case we need to recover more information to deal with the local stability of system (\ref{eq:errorlocal}). This is the objective of the following lemma, adapted from Lemma 1.6 in \cite{tarbouriech2011stability}, to deal with 
	$\psi_{v_{ref}}$.
	\begin{lemma} Given $r_l$ such that
		$v_{ref} > r_l > 0$. For any $\varepsilon \in \mathcal{S}(r_l) = \{\varepsilon \in \R^2 \ | \ - r_l \leq C \varepsilon \leq r_l\}$, the nonlinearity $\psi_{v_{ref}}$ satisfies the following inequality:
		\begin{equation}
		(\psi_{v_{ref}}(\varepsilon_1) + \Gamma_{v_{ref}} C \varepsilon)^\top (\psi_{v_{ref}}(\varepsilon_1) + \Gamma_{v_{ref}} C \varepsilon + \lambda C \varepsilon) \leq 0
		\label{eq:psisec}
		\end{equation}
		for any positive scalar $\lambda$ satisfying
		\begin{equation}
		\lambda \geq \lambda_{min}(r_l) > 0
		\label{eq:lambdamin1}
		\end{equation}
		with
		\begin{equation}
		\lambda_{min}(r_l) = \sup_{\varepsilon_1 \in [-r_l,0) \cup (0, r_l]} -\frac{\phi_{v_{ref}}(\varepsilon_1)}{\varepsilon_1} 
		\geq - \Gamma_{v_{ref}}.
		\label{eq:lambdamin2}
		\end{equation}
		\label{lem:psi}
	\end{lemma}
	
	\begin{proof} Recall that one gets $\phi_{v_{ref}}(\varepsilon_1) = \psi_{v_{ref}}(\varepsilon_1) + \Gamma_{v_{ref}} \varepsilon_1$. For any $\varepsilon \in \mathcal{S}(r_l)$, by definition it follows:
		\begin{equation}
		\left\{\begin{array}{lcl}
		\phi_{v_{ref}}(\varepsilon_1) > 0, \frac{\partial  \phi_{v_{ref}}}{\partial \varepsilon_1} < 0 & \mbox{ if} & - r_l \leq \varepsilon_1 < 0\\
		\phi_{v_{ref}}(\varepsilon_1) = 0, \frac{\partial  \phi_{v_{ref}}}{\partial \varepsilon_1} = \Gamma_{v_{ref}} & \mbox{ if} &  \varepsilon_1 = 0\\
		\phi_{v_{ref}}(\varepsilon_1) < 0, \frac{\partial  \phi_{v_{ref}}}{\partial \varepsilon_1} < 0 & \mbox{ if} & 0 < \varepsilon_1 \leq r_l
		\end{array}\right.
		\label{eq:phiproof}
		\end{equation}
		Hence, by taking the three cases of \eqref{eq:phiproof}, one gets the following:\\
		$\bullet$ Case 1: $- r_l \leq \varepsilon_1 < 0$. In this case the condition \eqref{eq:psisec} holds provided that $\phi_{v_{ref}}(\varepsilon_1) + \lambda \varepsilon_1 \leq 0$, which is satisfied if $\lambda \geq \max_{\varepsilon_1 \in [-r_l,0)} -\frac{\phi_{v_{ref}}(\varepsilon_1)}{\varepsilon_1}$.\\
		$\bullet$ Case 2: $\varepsilon_1 = 0$. In this case the condition \eqref{eq:psisec} holds for any positive $\lambda$ and therefore for any $\lambda$ satisfying \eqref{eq:lambdamin1}.\\
		$\bullet$ Case 3: $0 \leq \varepsilon_1 < r_l$. In this case the condition \eqref{eq:psisec} holds provided that $\phi_{v_{ref}}(\varepsilon_1) + \lambda \varepsilon_1 \geq 0$, which is satisfied if $\lambda \geq \max_{\varepsilon_1 \in (0,r_l]} -\frac{\phi_{v_{ref}}(\varepsilon_1)}{\varepsilon_1}$.
		
		Then from these three cases, one can guarantee that there exists $\lambda$ satisfying \eqref{eq:lambdamin1} such that \eqref{eq:psisec} holds. To compute the lower bound $\lambda_{min}(r_l)$ as described in \eqref{eq:lambdamin2}, one has to compute the supremum value of the function $g(\varepsilon_1) = - \frac{\phi_{v_{ref}}(\varepsilon_1)}{\varepsilon_1}$ on the interval $(-r_l,0) \cup (0, r_l)$. That corresponds to find the values for which its derivative is equal to zero. Therefore $\frac{\partial g(\varepsilon_1)}{\partial\varepsilon_1} = 0$ leads to $\phi_{v_{ref}}(\varepsilon_1) = \frac{\partial \phi_{v_{ref}}(\varepsilon_1)}{\partial\varepsilon_1} \varepsilon_1$. One can then check that 
		$\lambda_{min}(r_l)$ is defined as $$\lambda_{min}(r_l) = \sup_{\varepsilon_1 \in [-r_l,0) \cup (0, r_l]} -\frac{\phi_{v_{ref}}(\varepsilon_1)}{\varepsilon_1} \geq \lim_{\varepsilon_1 \rightarrow 0} -\frac{\phi_{v_{ref}}(\varepsilon_1)}{\varepsilon_1},$$ 
		with $ \lim_{\varepsilon_1 \rightarrow 0} -\frac{\phi_{v_{ref}}(\varepsilon_1)}{\varepsilon_1}= -\Gamma_{v_{ref}}$,
		by the definition of derivative. 
	\end{proof}
	
	\begin{remark} The previous lemma can be interpreted using the notions of co-coercivity and strongly monotone mapping, for more details, the reader can refer to \cite{bauschke2011convex}.
		\label{rem:coercivemap}
	\end{remark}
	\begin{remark}
		It is important to note that contrarily to the sector-condition for the dead-zone nonlinearity \cite{tarbouriech2011stability} where the slope of the nonlinearity at $0$ is $0$, in the current case the slope of the nonlinearity $\phi_{v_{ref}} (\varepsilon)$ at $0$ is $\Gamma_{v_{ref}} \neq 0$. This is the reason for which in Lemma \ref{lem:psi} there exists $\lambda_{min}(r_l) \neq 0$. 
	\end{remark}
	
	By using Lemma \ref{lem:psi}, the following theorem proposes a solution to Problem \ref{pb:regional}.
	
	\begin{theo} \label{theo:local} Given $v_{ref} \in (0,v_{ref 1}) \cup (v_{ref 2}, \infty)$, if there exist $r_l \in (0, v_{ref})$, $\lambda \in \R^+$, $P_l \in \mathbb{S}^2_+$ and $\tau > 0$ such that the following matrix inequalities hold:
		\begin{equation} \label{eq:LMIlocal}
		\Phi(v_{ref}) \prec 0,
		\end{equation}
		\begin{equation}
		\left[ \begin{array}{cc}
		P_l & \star \\ 
		C  & r_l^2 
		\end{array}\right] \succeq 0,
		\label{eq:inclu}    
		\end{equation}
		\begin{equation}
		\lambda \geq \lambda_{min}(r_l) > 0,
		\label{eq:lambdaMinLMI}    
		\end{equation}
		where 
		\begin{equation}
		\Phi(v_{ref}) = \left[ \begin{array}{cc}
		A_0^{\top} P_l + P_l A_0 - 2 \tau \Gamma_{v_{ref}}(\Gamma_{v_{ref}} + \lambda) C^\top C & \star \\ 
		B^\top P_l  -  \tau (2  \Gamma_{v_{ref}} + \lambda) C & -2 \tau 
		\end{array}\right],
		\end{equation}
		then the origin of system~\eqref{eq:error} is locally asymptotically stable and an estimation of the basin of attraction is:
		\[
		\D(P_l) = \left\{ \varepsilon \in \mathbb{R}^2 \ | \ \varepsilon^{\top} P_l \varepsilon \leq 1 \right\}.
		\]
		Hence, the set $\D(P_l)$ is a solution to Problem \ref{pb:regional}.
	\end{theo}
	
	\begin{proof} Consider the following Lyapunov function:
		\begin{equation}
		V_l(\varepsilon) = \varepsilon^{\top} P_l \varepsilon,
		\end{equation}
		with $P_l \in \S^2_+$. Then, one wants to ensure that
		$\langle \nabla V_l,f \rangle < 0$, for any $\varepsilon \in \D(P_l)$, for any $f \in A_0 \varepsilon(t)  +  B \psi_{v_{ref}}(\varepsilon_1)$ and for all nonlinearities $\psi_{v_{ref}}(\varepsilon_1)$ satisfying Lemma \ref{lem:psi}. 
		Similarly to the proof of Theorem \ref{theo:global}, we denote by language abuse $\langle \nabla V_l,f \rangle$ by $ \dot{V}_l(\varepsilon)$ meaning that we compute $\langle \nabla V_l,f \rangle$ for any solution $f$ to the differential equation \eqref{eq:errorlocal2}.
		
		First, let us observe that the satisfaction of relation \eqref{eq:inclu} is equivalent to $P_l \succeq C^{\top} C r_l^{-2}$ by Schur complement. Consequently for any $\varepsilon \in \D(P_l)$, we get $(C \varepsilon)^2 \leq r_l^2$ which implies that $\varepsilon \in \mathcal{S}(r_l)$ (defined in Lemma \ref{lem:psi}). Therefore, for any $\varepsilon \in \D(P_l)$, Lemma \ref{lem:psi} applies and for any positive $\lambda$ satisfying \eqref{eq:lambdaMinLMI},  condition~\eqref{eq:psisec} holds.
		
		Note that for all solution to \eqref{eq:errorlocal2}, one gets: 
		\begin{equation}
		\dot{V}_l(\varepsilon) = \varepsilon^\top (A_0^\top P_l + P_l A_0) \varepsilon + 2 \varepsilon^\top P_l B \psi_{v_{ref}}(\varepsilon_1).
		\end{equation}
		Thanks to the satisfaction of relation \eqref{eq:inclu}, by using condition~\eqref{eq:psisec} of Lemma \ref{lem:psi}, it follows that for $\tau > 0$ and $\varepsilon \in \D(P_l)$:
		\begin{equation}
		\dot{V}_l(\varepsilon) \leq \dot{V}_l(\varepsilon) - 2 \tau (\psi_{v_{ref}}(\varepsilon_1) + \Gamma_{v_{ref}} C \varepsilon)^\top (\psi_{v_{ref}}(\varepsilon_1) + \Gamma_{v_{ref}} C \varepsilon + \lambda C \varepsilon).
		\end{equation}
		Hence, to ensure that $\dot{V}_l(\varepsilon) < 0$, it suffices to ensure  $\mathcal{L} = \dot{V}_l(\varepsilon) - 2 \tau (\psi_{v_{ref}}(\varepsilon_1) + \Gamma_{v_{ref}} C \varepsilon)^\top (\psi_{v_{ref}}(\varepsilon_1) + \Gamma_{v_{ref}} C \varepsilon + \lambda C \varepsilon) < 0$. $\mathcal{L}$ reads $\mathcal{L} \in \xi^\top \Phi(v_{ref}) \xi$
		where $\xi = \displaystyle \col \left( \varepsilon, \psi_{v_{ref}}(\varepsilon_1) \right)$.
		
		Hence, if relation \eqref{eq:LMIlocal} holds, it follows that $\mathcal{L} < 0$ and therefore $\dot{V}_l$ is negative definite. 
		If the conditions of the theorem are verified, then for any initial condition $\varepsilon(0) \in \D(P_l)$, we get that $\varepsilon(t) \in \D(P_l)$ and therefore the trajectories of system~\eqref{eq:error} converge to the origin. One can conclude that the origin of system \eqref{eq:errorlocal2} is regionally asymptotically stable in $\D(P_l)$, or equivalently, \eqref{eq:error} is regionally asymptotically stable in $\D(P_l)$.\end{proof}
	
	\begin{remark}
		From Lemma \ref{lem:psi}, one gets  $\lambda > - \Gamma_{v_{ref}}$. Then, one can deduce that a necessary condition for the feasibility of \eqref{eq:LMIlocal} is that $A_0$ is Hurwitz, which holds since we consider $v_{ref} \in (0,v_{ref 1}) \cup (v_{ref 2}, \infty)$ (see Lemma \ref{lem:A0}).
		\label{rem:lambdaGamma}
	\end{remark}
	
	Regarding the feasibility of conditions \eqref{eq:LMIlocal}-\eqref{eq:lambdaMinLMI}, the following corollary can be stated.
	\begin{coro} \label{coro:existence} 
		\phantom{a}
		\begin{enumerate}
			\item Given $v^{\circ} \geq v_{ref2}$, if there exist $\lambda$, $r_l \in (0, v^{\circ})$, $P_l \in \mathbb{S}^2_+$ and $\tau > 0$ such that $\Phi(v^{\circ}) \prec 0$ and $\Phi(\infty) \prec 0$ together with \eqref{eq:inclu}-\eqref{eq:lambdaMinLMI}, then system~\eqref{eq:error} is locally asymptotically stable for any $v_{ref} \geq v^{\circ}$ with an inner-estimation of the basin of attraction $\D(P_l)$.
			\item If $A$, defined in \eqref{eq:A}, is Hurwitz, the matrix inequalities \eqref{eq:LMIlocal}-\eqref{eq:lambdaMinLMI} are feasible.
		\end{enumerate}
	\end{coro}
	\begin{proof} \phantom{a}
		
		\textbf{First item:} Let first note that the left-hand side in relation \eqref{eq:LMIlocal} can be rewritten as follows:
		\begin{equation}
		\begin{array}{lcl}
		\Phi(v_{ref}) & = & \left[ \begin{array}{cc}
		A^{\top} P_l + P_l A & \star \\ 
		B^\top P_l - \tau \lambda C & -2 \tau 
		\end{array}\right] + \Gamma_{v_{ref}} 
		\left[ \begin{array}{cc}
		\He(P_l B C) 
		- 2 \tau \lambda C^\top C & \star \\ 
		- 2 \tau C & 0  
		\end{array}\right] \\
		& &
		+ \Gamma_{v_{ref}}^2 
		\left[ \begin{array}{cc}
		- 2 \tau C^\top C & \star \\ 
		0 & 0  
		\end{array}\right]
		\end{array}
		\label{eq:defMfeasibility}
		\end{equation}
		That shows that $\Phi$ is convex in $\Gamma_{v_{ref}}$. Furthermore, note that $\Gamma_{v_{ref}}$ is strictly decreasing with respect to $v_{ref}$ and consequently, $\Phi(v_{ref})$ is convex in $v_{ref}$ for $v_{ref} > v_{ref2}$ and that ends the proof for the first item.
		
		\textbf{Second item:} The proof of the second item takes advantage of the proof of the first one. Given $\lambda, r_l$ such that \eqref{eq:lambdaMinLMI} holds, from \eqref{eq:defMfeasibility}, it follows:
		\begin{equation}
		\Phi(v_{ref}) = \Phi(\infty) + \Gamma_{v_{ref}} 
		\He\left(\left[ \begin{smallmatrix}
		P_l B - \tau (\Gamma_{v_{ref}} + \lambda) C^\top\\ 
		- \tau    
		\end{smallmatrix}\right]  \left[\begin{smallmatrix}
		C & 0   
		\end{smallmatrix}\right]\right)
		\label{eq:defMfeasibility2}
		\end{equation}
		Denote by $\mathcal{N}_1$ and $\mathcal{N}_2$ the basis of the Kernel of $\left[\begin{smallmatrix}
		C & 0   
		\end{smallmatrix}\right]$ and $\left[\begin{smallmatrix}
		B^\top P_l  - \tau (\Gamma_{v_{ref}} + \lambda) C & 
		- \tau    
		\end{smallmatrix}\right]$, respectively.
		By using the Elimination Lemma \cite{LMI}, the satisfaction of $\Phi(v_{ref}) \prec 0$ is equivalent from \eqref{eq:defMfeasibility2} to:
		\begin{equation}
		\begin{array}{ccc}
		\mathcal{N}_1^\top \Phi(\infty) \mathcal{N}_1 \prec 0, & \quad &
		\mathcal{N}_2^\top \Phi(\infty) \mathcal{N}_2 \prec 0.
		\end{array}
		\label{eq:defMfeasibility3}
		\end{equation} 
		Furthermore, note that $\Phi(\infty)$ can be written as follows:
		\begin{equation}
		\Phi(\infty) =  \left[ \begin{array}{cc}
		A^{\top} P_l + P_l A & 0 \\ 
		0 & 0 
		\end{array}\right] + \He\left(\left[ \begin{smallmatrix}
		P_l B - \tau \lambda C^\top\\ 
		- \tau    
		\end{smallmatrix}\right]  \left[\begin{smallmatrix}
		0 & I   
		\end{smallmatrix}\right]\right)
		\label{eq:defMfeasibility4}
		\end{equation}
		By using the Elimination Lemma \cite{LMI}, the satisfaction of $\Phi(\infty) \prec 0$ is equivalent to $A^{\top} P_l + P_l A \prec 0$. Hence, if $A$ is Hurwitz, then there exists $P_l$ such that $\Phi(\infty) \prec 0$ and consequently,  \eqref{eq:defMfeasibility3} holds. Since $k \cdot P_l$ for $k \geq 1$ is also solution, there exists $k$ large enough such that \eqref{eq:inclu} holds. Therefore relations \eqref{eq:LMIlocal}-\eqref{eq:lambdaMinLMI} also hold and the proof is completed.
	\end{proof}
	
	\begin{remark}
		The first assertion of Corollary~\ref{coro:existence} implies that the basins of attraction have a minimal axis which is growing with $v_{ref}$.
	\end{remark}
	
	
	
	
	
	At this point, the aim is to derive an inner-estimate of the basin around the origin which is close to the real basin. To this extent, one can decide to solve the following optimization problem:
	\begin{equation} \label{eq:optBasin}
	\begin{array}{cc}
	\min & \eta \\
	\text{subject to} & \eqref{eq:LMIlocal}-\eqref{eq:lambdaMinLMI} \text{ hold with } P_l \preceq \eta I_2.
	\end{array}
	\end{equation}
	If $r_l$ is a decision variable, this is not a semi-definite optimization problem since $\lambda_{min}$ is nonlinearly related to $r_l$. Indeed, the relation \eqref{eq:LMIlocal} is an LMI provided that $\tau$ or $\lambda$ is fixed. Moreover $\lambda$ must satisfy condition \eqref{eq:lambdaMinLMI}, which depends on $r_l$. A way to have a set of LMI for \eqref{eq:LMIlocal} consists in fixing $r_l$, allowing to compute $\lambda_{min}(r_l)$, and to consider an additional variable $\gamma = \tau \lambda$ with the constraints $\gamma > \tau \lambda_{min}(r_l)$. \\
	From this remark, we derive the following algorithm which aims at providing a sub-optimal solution\footnote{The code of this algorithm is available at \url{https://github.com/mBarreau/MassSpringSystem}.}:
	\begin{enumerate}
		\item Let $r_l = v_{ref}$, $\eta_{min} = \infty$ and $P_{min} = 0$;
		\item Compute $\lambda_{min}(r_l)$ using
		\eqref{eq:lambdamin2};
		\item Solve \eqref{eq:optBasin} using the previous paragraph:
		\begin{itemize}
			\item if the problem is feasible, then 
			\begin{itemize}
				\item if $\eta < \eta_{min}$, then
				\begin{enumerate}
					\item $\eta_{min} \leftarrow \eta$ and $P_{min} \leftarrow P_l$;
					\item use a dichotomy algorithm to decrease $r_l$;
					\item go back to step 2;
				\end{enumerate}
				\item otherwise stop.
			\end{itemize}
			\item if the problem is not feasible, then
			\begin{enumerate}
				\item decrease slightly $r_l$;
				\item go back to step 2;
			\end{enumerate}
		\end{itemize}
	\end{enumerate}
	Due to the second assertion in Corollary~\ref{coro:existence}, this algorithm will stop for $r_l$ small enough. The result of this algorithm is a matrix $P_{min}$ such that $\D(P_{min})$ is a basin of attraction with a minimal axis which is a local maximum.
	
	Combining the results of the previous parts leads to the main theorem of this article which is derived in the following section.
	
	\section{Global Asymptotic Stability Analysis}
	\label{sec:GAS}
	
	The main idea here is to use at the same time Theorems~\ref{theo:local} and \ref{theo:global}to get a condition on $v_{ref}$ to guarantee that the origin is globally asymptotically stable for system \eqref{eq:error}, and therefore to provide a solution to Problem~\ref{pb:main}.
	
	\begin{theo} \label{theo:GAS}
		Assume $v_{ref} \in (0,v_{ref 1}) \cup (v_{ref 2}, \infty)$ and $A$ (defined in equation~\eqref{eq:A}) Hurwitz. If there exist $\lambda$, $r_l \in (0, v_{ref})$, $\tau$, $\tau_0$, $\tau_1$, $\tau_2$, $\tau_3$, $\tau_4$, $\tau_5$ and $P_l, P_g \succ 0$ satisfying relations \eqref{eq:LMIglobal1}, \eqref{eq:LMIglobal2}, \eqref{eq:LMIlocal}, \eqref{eq:inclu} and
		\begin{equation}
		P_g - P_l \succeq 0
		\label{eq:inclu2}
		\end{equation}
		then $\A(P_g) \subseteq \D(P_l)$ and 
		the origin of \eqref{eq:error} is globally asymptotically stable.
	\end{theo}
	
	\begin{proof} Assuming $v_{ref} \in (0,v_{ref 1}) \cup (v_{ref 2}, \infty)$ and $A$  Hurwitz implies that there exist $P_g, P_l \succ 0$ such that the relations \eqref{eq:LMIglobal1}, \eqref{eq:LMIglobal2}, \eqref{eq:LMIlocal}, \eqref{eq:inclu} are feasible. Consequently, Theorems \ref{theo:local} and \ref{theo:global} apply. It allows to characterize an attractor $\A(P_g)$ towards which all the trajectories of system \eqref{eq:error} converge and an estimation of the basin of attraction of the origin $\D(P_l)$, from which the trajectories of \eqref{eq:error} converges to the origin. Condition \eqref{eq:inclu2} guarantees the inclusion $\A(P_g) \subseteq \D(P_l)$. Then by definition of both sets, one can conclude that when the trajectories converging towards $\A(P_g)$ enter in $\D(P_l)$, they finally converge to the origin. That concludes the proof on the global asymptotic stability of the origin. \end{proof}
	
	\section{Numerical Simulations}
	\label{sec:ex}
	
	This section is dedicated to numerical results with the parameters defined as in Table~\ref{tab:parameters}. 
	
	\begin{table}[!hhh]
		\centering
		\begin{tabular}{c|ccccccccc}
			Parameter & 
			$m$ & $g$ & $v_s$ & $\mu_C$ & $\mu_S$ & $k$ & $k_v$ & $\ell_0$ & $x_A(0)$ \\ \hline
			Value & $1$ &  $9.81$ & $0.8$ & $0.2997$ &  $0.5994$ & $2$ & $1$ & $0$ & $0$
		\end{tabular}
		\caption{Parameters used for the simulations (taken from \cite{johanastrom2008revisiting}).}
		\label{tab:parameters}
	\end{table}
	
	
	The solver used for the LMIs is Mosek \cite{mosek} together with YALMIP \cite{1393890}. The simulations are conducted using a mixed Newton-backward, Newton-forward discretization scheme adapted for set-valued functions as proposed in \cite{acary2008numerical}. This adaptation prevents numerical chattering and is available at \url{https://github.com/mBarreau/MassSpringSystem}.
	\begin{figure}
		\centering
		\includegraphics[width=8cm]{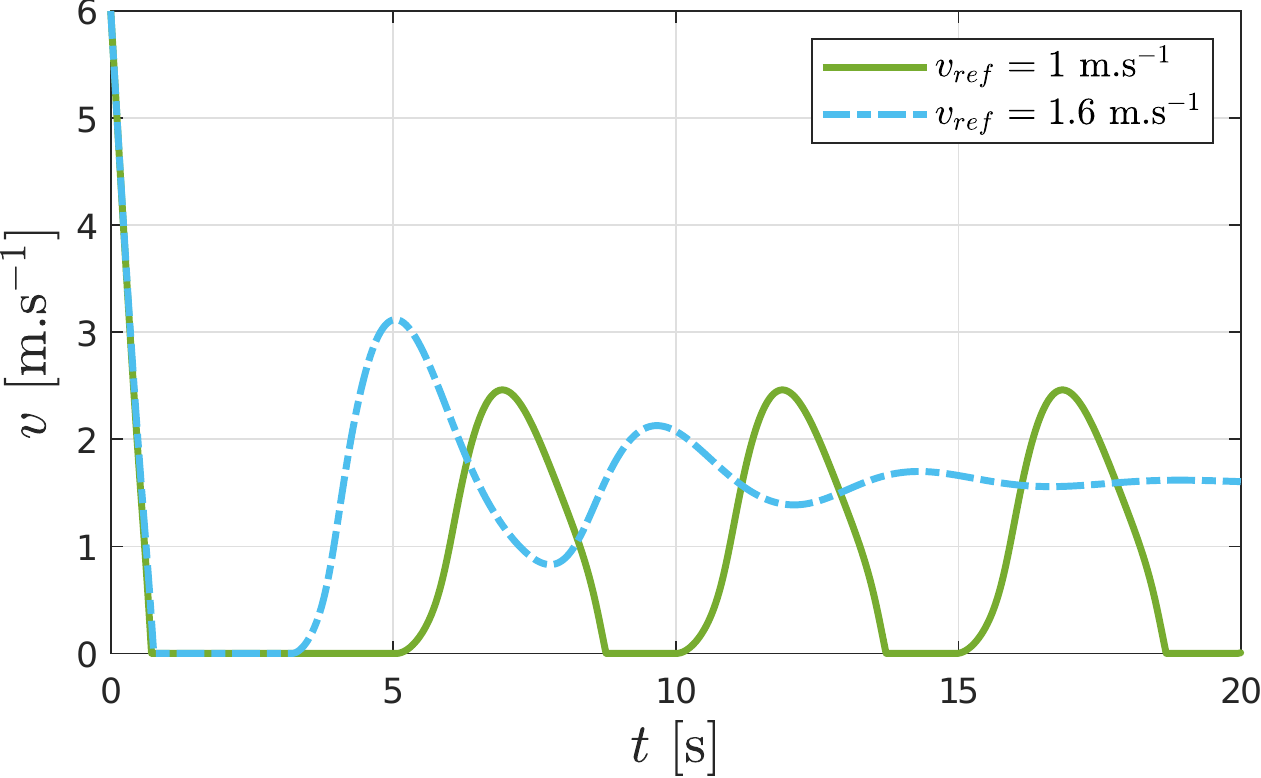}
		\caption{Numerical simulation of \eqref{eq:stateSpace} for $v_{ref} \in \{1, 1.6\}$ m.s${}^{-1}$, $v_0 = 6$ m.s${}^{-1}$ and $z_0 = 0$ m.}
		\label{fig:simu}
	\end{figure}
	
	Figure~\ref{fig:simu} depicts a numerical simulation of \eqref{eq:stateSpace} with  the parameters in Table~\ref{tab:parameters} for two different reference speeds. The stick-slip phenomenon is clearly visible since after five seconds for $v_{ref} = 1$ m.s${}^{-1}$, a cycle of period $5s$ and of amplitude $2.4$m.s${}^{-1}$ is emerging. For a reference speed higher than $1.45$ m.s${}^{-1}$, it seems that the stick-slip phenomenon disappears. In \cite{canudasdewit:hal-00394990}, the authors are using a describing function analysis on a similar system, in our case, we would get that for $v_{ref} \geq 1.51$ m.s${}^{-1}$ there does not exist any limit cycle. This seems a very good estimation but it is not a rigorous analysis.
	
	\begin{figure*}
		\centering
		\begin{subfigure}[t]{0.45\linewidth}
			\includegraphics[width=\textwidth]{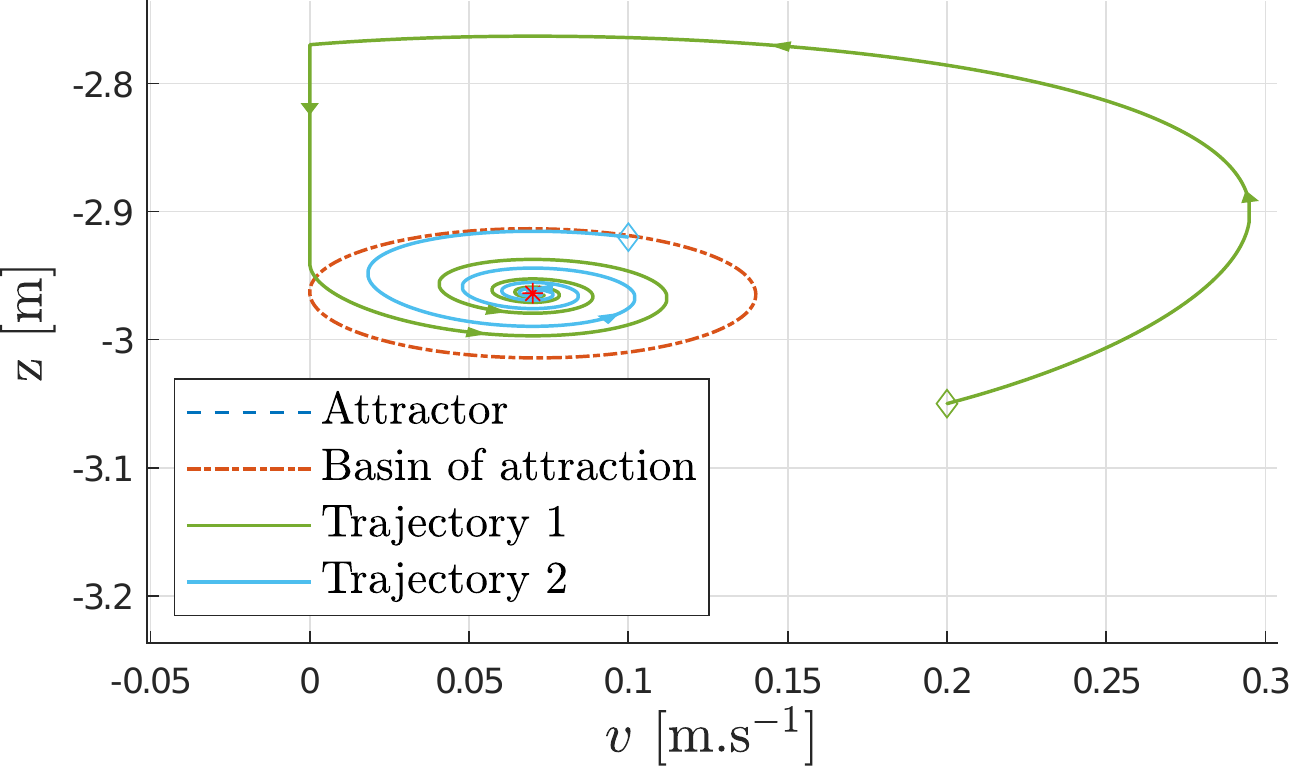}
			\caption{$v_{ref} = 0.07$ m.s${}^{-1}$}
			\label{fig:phasePlanevref007}
		\end{subfigure} \
		\begin{subfigure}[t]{0.45\linewidth}
			\includegraphics[width=\textwidth]{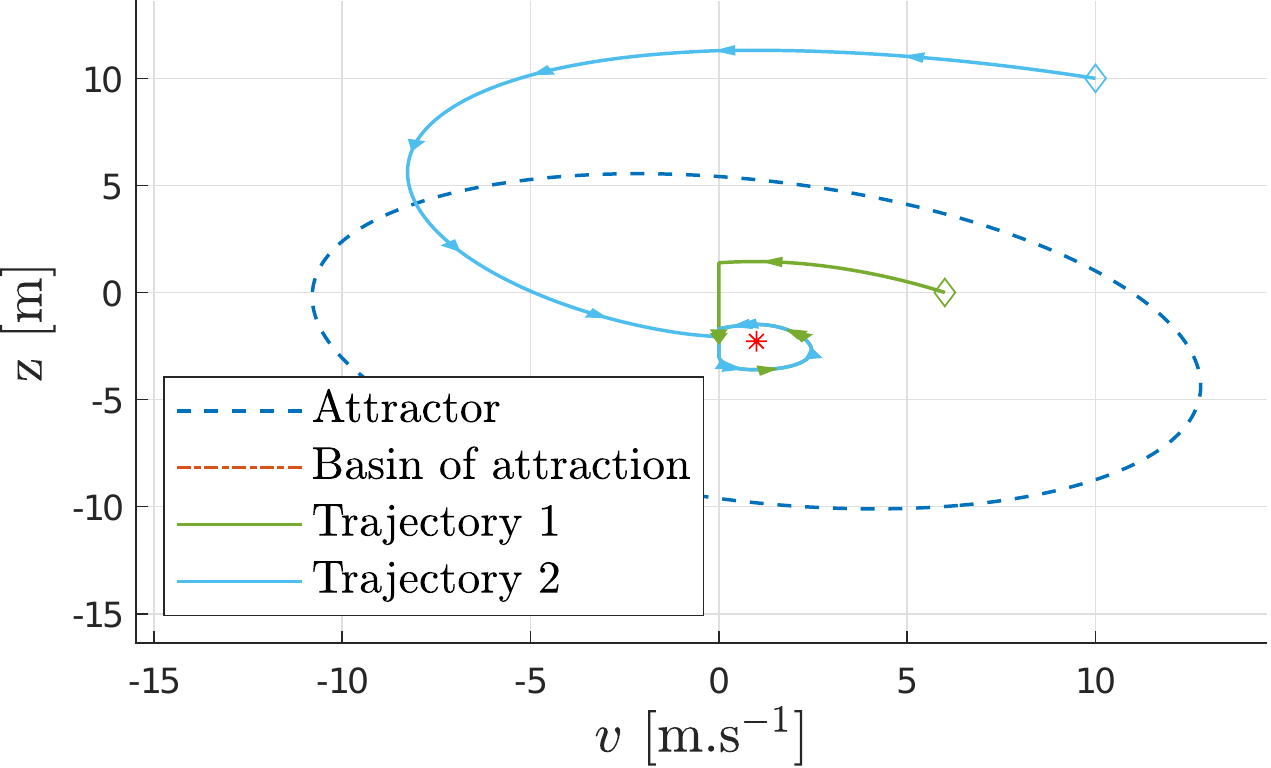}
			\caption{$v_{ref} = 1$ m.s${}^{-1}$}
			\label{fig:phasePlanevref1}
		\end{subfigure} \\
		\begin{subfigure}[t]{0.45\linewidth}
			\includegraphics[width=\textwidth]{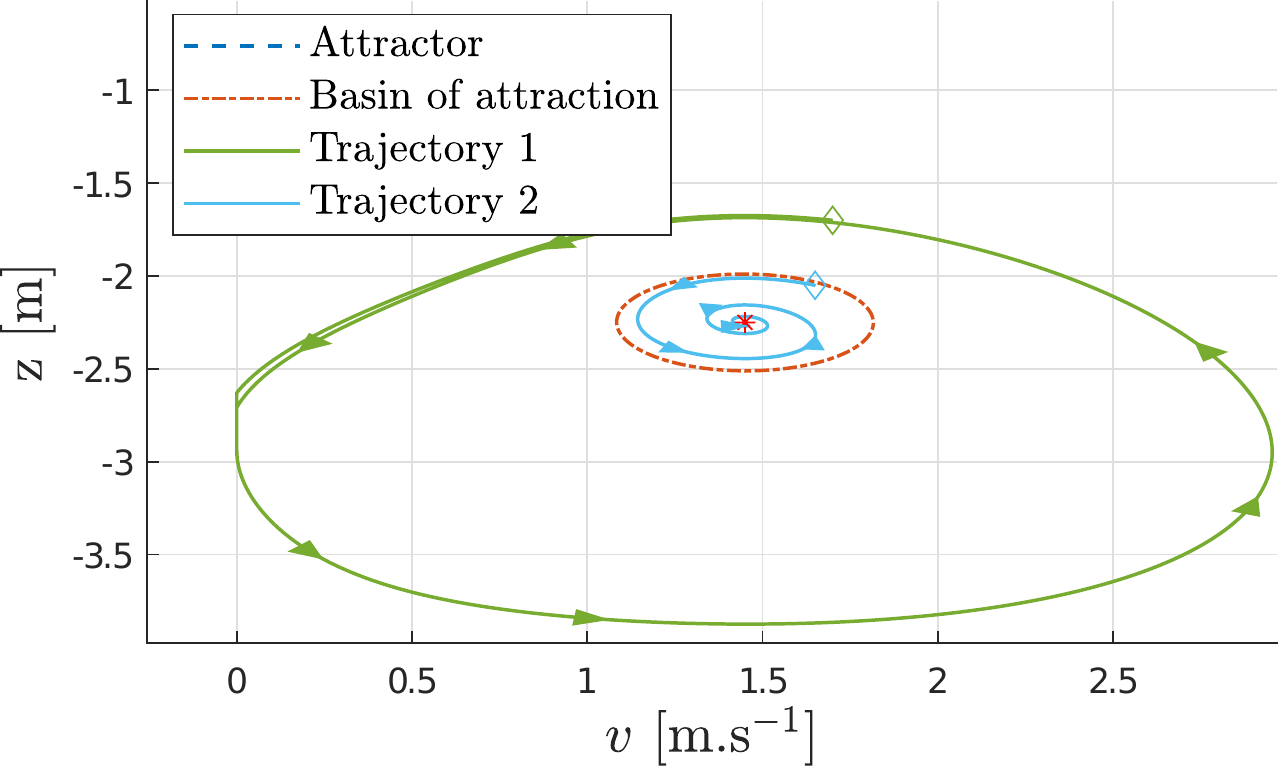}
			\caption{$v_{ref} = 1.45$ m.s${}^{-1}$}
			\label{fig:phasePlanevref1.45}
		\end{subfigure} \
		\begin{subfigure}[t]{0.45\linewidth}
			\includegraphics[width=\textwidth]{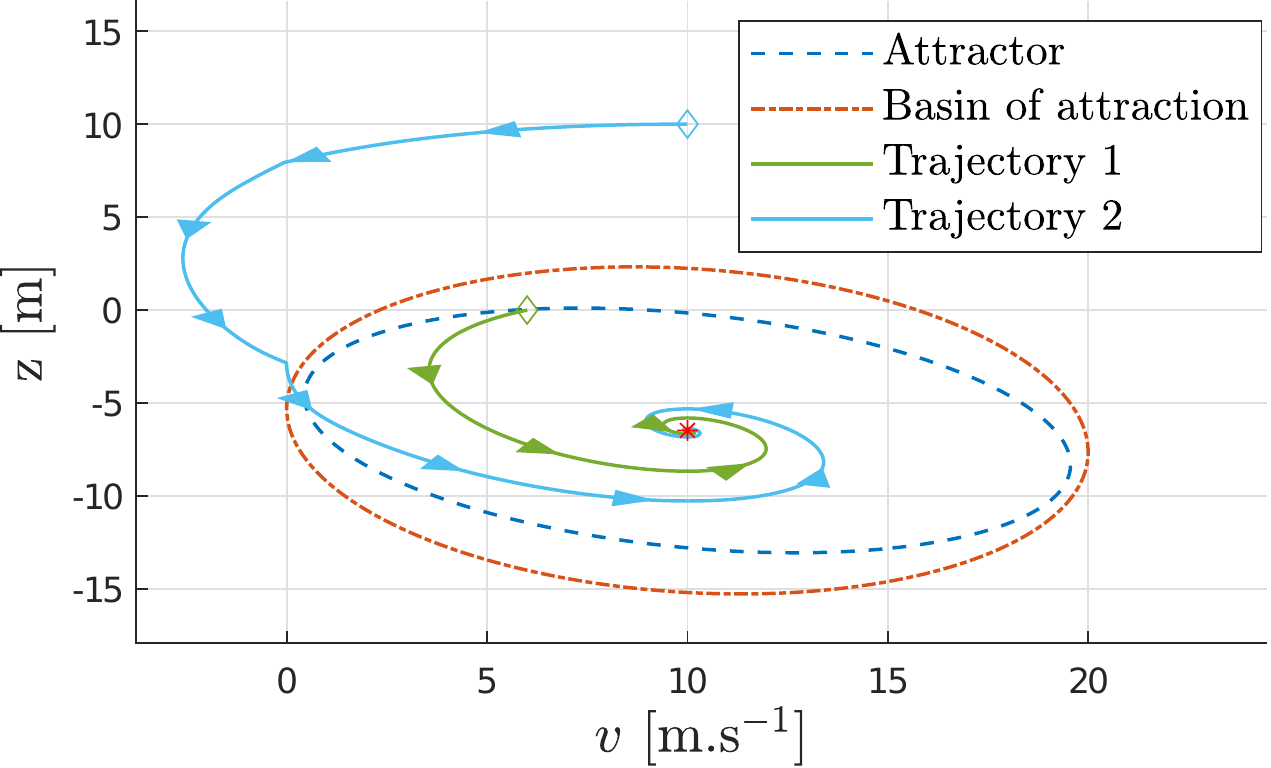}
			\caption{$v_{ref} = 10$ m.s${}^{-1}$}
			\label{fig:phasePlanevref10}
		\end{subfigure} \\
		\caption{Phase portraits of \eqref{eq:stateSpace} for four values of $v_{ref}$. The inner-approximation of $\A$ is in dash-blue and the outer estimation of $\D$ is in dash-dot-red. The equilibrium point is the red star and the two trajectories initiate from the diamond points.}
		\label{fig:phasePlan}
	\end{figure*}
	
	Using the methodology developed in this paper, one can compute the numerical values defining the interval of admissible $v_{ref}$, that is: $v_{ref 1} = 0.11$ m.s${}^{-1}$ and $v_{ref 2} = 1.25$ m.s${}^{-1}$. Then, for $v_{ref} \in (0,v_{ref 1}) \cup (v_{ref 2}, \infty)$, there exists a basin of attraction around the origin. Using Theorem~\ref{theo:GAS}, we find that for $v_{ref} \geq 9.59$ m.s${}^{-1}$, the equilibrium point is globally asymptotically stable\footnote{Using the method of \cite{canudasdewit:hal-00394990}, one gets a smaller value for global asymptotic stability. Nevertheless, it does not estimate the basin of attraction and the attractor.}. 
	As noted before, there always exists an attractor for the system but the existence or not of a basin of attraction leads to four scenarios:
	\begin{enumerate}
		\item $v_{ref} \in (0, 0.11)$ m.s${}^{-1}$: there exists a basin of attraction around the equilibrium point;
		\item $v_{ref} \in [0.11, 1.25]$ m.s${}^{-1}$: the equilibrium point is not asymptotically stable;
		\item $v_{ref} \in (1.25, 9.59)$ m.s${}^{-1}$: there exists a basin of attraction which does not include the attractor; 
		\item $v_{ref} \geq 9.59$ m.s${}^{-1}$: the basin of attraction contains the attractor and the equilibrium point is globally asymptotically stable.
	\end{enumerate}
	
	Figure~\ref{fig:phasePlan} shows the result of simulations in these four cases. The attractor\footnote{Sometimes, the attractor is not displayed since it is not of the same order of magnitude as the basin of attraction.} is computed with Theorem~\ref{theo:global} and the basin of attraction with Theorem~\ref{theo:local}. \\
	In Figure~\ref{fig:phasePlanevref007}, one can see that the trajectories are converging even if they start far away from the basin of attraction. Simulations tend to show that the equilibrium point is globally exponentially stable but our analysis does not reflect that. \\
	In the second case (Figure~\ref{fig:phasePlanevref1}), there does not exist a basin of attraction. The attractor contains the oscillations but is very large compared to the real oscillations. The conservatism might come from the use of quadratic Lyapunov functions and a relatively rude encapsulation of the nonlinear friction term (see Figure~\ref{fig:relay}).\\
	Figure~\ref{fig:phasePlanevref1.45} shows that the inner-approximation of the basin of attraction is good and for a trajectory initiated close but outside this estimation, the stick-slip phenomenon occurs. The equilibrium point is not globally asymptotically stable but it stays locally asymptotically stable.\\
	Finally, in the last case (Figure~\ref{fig:phasePlanevref10}), the attractor is included in the inner-approximation of the basin of attraction and the equilibrium point is globally asymptotically stable. 
	
	
	\section{Conclusion}
	\label{sec:conclu}
	
	This paper proposes three theorems dealing with the stability of a system subject to friction. The three proposed methods give a characterization of the global attractor of the system, an estimation of the basin of attraction of the origin when it is possible, and finally a global asymptotic stability condition. The conditions are given in terms of LMIs or quasi-LMIs\footnote{To be understood in the sense that the nonlinear term is issued from a product between a scalar and a matrix.} with efficient algorithms to solve them. 
	
	As depicted in the numerical simulations section, the reference speed for which global asymptotic stability exists is overestimated mainly because of the conservatism introduced in the global stability test. The first direction of research would be to estimate an attractor that might not be symmetrical and to study more precisely the friction function $F_{nl}$ to get a better bounding. In this context, one idea could be to use the dissipated energy by friction,  which is always positive but not smooth, to design a Lyapunov function. Moreover, taking inspiration from the techniques proposed in \cite{leine2007stability}, \cite{brogliato2007dissipative} the notion of dissipativity could be interesting when dealing not only with the global attractor but also with the regional stability.  Another perspective would be to extend to higher-order problems (by considering a PID controller for instance) and real-life applications.
	
	\bibliography{mybibfile}
	
\end{document}